\documentclass[reqno]{amsart}
\usepackage{amscd}
\usepackage{amssymb}
\usepackage[margin=1in,includeheadfoot]{geometry}
\usepackage{cite}
\usepackage [latin1]{inputenc}
\usepackage{tabularx,booktabs,tikz}
\usepackage{caption}
\usepackage{amsmath}
\usepackage{amsfonts}

\usepackage{amscd}
\usepackage{amsthm}
\usepackage{amssymb} \usepackage{latexsym}
\usepackage{eufrak}
\usepackage{euscript}
\usepackage{epsfig}
\usepackage{graphics}
\usepackage{array}
\usepackage{enumerate}
\usepackage{dsfont}
\usepackage{color}
\usepackage{wasysym}
\usepackage{hyperref}
\usepackage{pdfsync}

\def\beq{\begin{equation}}
\def\eeq{\end{equation}}
\def\ba{\begin{array}}
\def\ea{\end{array}}

\def\R{\mathbb R}

\newtheorem{thm}{Theorem}[section]
\newtheorem{lm}[thm]{Lemma}

\theoremstyle{definition}

\theoremstyle{remark}

\begin{document}
\pagestyle{plain}
\title{Solutions to discrete nonlinear Kirchhoff-Choquard equations}

\author{Lidan Wang}
\email{wanglidan@ujs.edu.cn}
\address{Lidan Wang: School of Mathematical Sciences, Jiangsu University, Zhenjiang 212013, People's Republic of China}

\begin{abstract}
In this paper, we study the discrete Kirchhoff-Choquard equation
$$
-\left(a+b \int_{\mathbb{Z}^3}|\nabla u|^{2} d \mu\right) \Delta u+V(x) u=\left(R_{\alpha} *F(u)\right)f(u),\quad x\in \mathbb{Z}^3,
$$
where $a,\,b>0$ are constants, $R_{\alpha}$ is the Green's function of the discrete fractional Laplacian with $\alpha \in(0,3)$, which has no singularity but has same asymptotics as the Riesz potential. Under some suitable assumptions on $V$ and $f$, we prove the existence of nontrivial solutions and ground state solutions by variational methods.
\end{abstract}

\maketitle


\
\


\section{Introduction}

The Kirchhoff-type equation
\begin{equation}\label{z}
    -\left(a+b \int_{\mathbb{R}^3}|\nabla u|^{2} d \mu\right) \Delta u+V(x) u=g(x,u),\quad u\in H^1(\mathbb{R}^3),
\end{equation}
where $a,\,b>0$, has drawn lots of interest in recent years due to the appearance of the term $\int_{\mathbb{R}^3}|\nabla u|^{2} d \mu$. For example, Wu \cite{W} proved the existence of nontrivial solutions under general assumptions on $g$ by 
the symmetric mountain pass theorem.  If $g(x,u)=g(u)$, He and Zou \cite{HZ} showed the existence of ground state solutions under the Ambrosetti-Rabinowitz
conditions on $g$ by the Nehari manifold approach; Guo \cite{G} also derived the existence of ground state solutions when $g$ does not satisfy the Ambrosetti-Rabinowitz
conditions; Wu and Tang \cite{WT} verified the existence and concentration of ground state solutions under some assumptions on $V$ and $g$ by the sign-changing Nehari manifold method. If $g(x,u)=|u|^{p-1}u$, Sun and Zhang \cite{SZ} obtained the
uniqueness of ground state solutions for $p\in(3,5)$. By a monotonicity trick and a new version of global compactness lemma, Li and Ye \cite{LY} established the existence of ground state solutions for $p\in(2,5)$. Later, Lv and Lu \cite{LL} extended this result  for all $p\in(1,5)$ by different methods. For more related works, we refer the readers to \cite{CT,CW,HQ,TS}.

On the other hand, in many physical applications, the Choquard-type nonlinearity appears naturally, namely $g(x,u)=\left(I_\alpha \ast F(u)\right)f(u)$, where $I_\alpha$ is the Riesz potential. Clearly, two nonlocal terms are involved in the equation (\ref{z}), which makes it more difficult to deal with. Recently, for $\alpha\in (1, 3)$, Zhou and Zhu \cite{ZZ1} proved the existence of ground state solutions; Liang et al. \cite{LS} obtained the existence of multi-bump solutions. For $\alpha\in(0,3)$, Chen, Zhang and Tang \cite{CZ} proved the existence of ground state solutions under some hypotheses on $V$ and $f$; Lv and Dai \cite{LD} established the existence and asymptotic behavior of ground state solutions by a Pohozaev-type constraint technique; Hu et al. \cite{H} obtained two classes of ground state solutions under the general Berestycki-Lions conditions on $f$. If $f(u)=|u|^{p-2}u$ with $p\in(2,3+\alpha)$, Lv \cite{L} demonstrated the existence and asymptotic behavior 
of ground state solutions by the
Nehari manifold and the concentration compactness principle.
For more related works, we refer the readers to \cite{GT,HT,LP,LL,YG}.

Nowadays, many researchers  turn to study  differential equations on graphs, especially for the nonlinear elliptic equations. See for examples \cite{GLY,HSZ,HLW,HX,W2,ZZ} for the  discrete nonlinear  Schr\"{o}diner equations. For the  discrete nonlinear Choquard equations, we refer the readers to \cite{LW,LZ1,LZ,W1}. Recently, Lv \cite{L1} proved the existence of ground state solutions for a class of Kirchhoff equations on lattice graphs $\mathbb{Z}^3$. To the best of our knowledge, there is no existence results for the Kirchhoff-Choquard equations on graphs. Motivated by the works mentioned above, in this paper, we would like to study a class of Kirchhoff-type equations with general convolution nonlinearity on lattice graphs $\mathbb{Z}^3$ and discuss the existence of solutions under different conditions on potential $V$.

Let us first give some notations. Let $C(\mathbb{Z}^{3})$ be the set of all functions on $\mathbb{Z}^{3}$ and $C_{c}\left(\mathbb{Z}^{3}\right)$ be the set of all functions on $\mathbb{Z}^{3}$ with finite support. We denote by the $\ell^p(\mathbb{Z}^3)$ the space of $\ell^p$-summable functions on $\mathbb{Z}^3$. Moreover, for any $u\in C(\mathbb{Z}^{3})$, we always write $
\int_{\mathbb{Z}^{3}} f(x)\,d \mu=\sum\limits_{x \in \mathbb{Z}^{3}} f(x),
$
where $\mu$ is the counting measure in $\mathbb{Z}^{3}$.

In this paper, we consider the following Kirchhoff-Choquard equation
\begin{equation}\label{aa}
-\left(a+b \int_{\mathbb{Z}^3}|\nabla u|^{2} d \mu\right) \Delta u+V(x) u=\left(R_{\alpha} *F(u)\right)f(u),\quad x\in \mathbb{Z}^3, 
\end{equation}
where $a, b>0$ are constants, $\alpha\in(0,3)$ and $R_{\alpha}$ represents the Green's function of the discrete fractional Laplacian, see \cite{MC,W1}, $$ R_{\alpha}(x,y)=\frac{K_{\alpha}}{(2\pi)^3}\int_{\mathbb{T}^3}e^{i(x-y)\cdot k}\mu^{-\frac{\alpha}{2}}(k)\,dk,\quad x,y\in \mathbb{Z}^3,$$
which contains the fractional degree
$$K_\alpha=\frac{1}{(2\pi)^3}\int_{\mathbb{T}^3}\mu^{\frac{\alpha}{2}}(k)\,dk,\,\mu(k)=6-2\sum\limits_{j=1}^3 \cos(k_j),$$
where $\mathbb{T}^3=[0,2\pi]^3,\,k=(k_1,k_2,k_3)\in\mathbb{T}^3.$  Clearly, the Green's function $R_{\alpha}$ has no singularity at $x=y$. According to \cite{MC}, the Green's function $R_\alpha$  behaves as $x-y|^{\alpha-3}$ for $|x-y|\gg 1$. 
Here $\Delta u(x)=\underset {y\sim x}{\sum}(u(y)-u(x))$ and $|\nabla u(x)|=\left(\frac{1}{2} \sum\limits_{y \sim x}(u(y)-u(x))^{2}\right)^{\frac{1}{2}}.$ 

Now we give assumptions on the potential $V$ and the nonlinearity $f$:
\begin{itemize}
\item[($h_1$)] for any $x\in\mathbb{Z}^3$, there exists $V_0>0$ such that $V(x) \geq V_0$;
\item[($h_2$)] there exists a point $x_0\in\mathbb{Z}^3$ such that $V(x)\rightarrow\infty$ as $|x-x_0|\rightarrow\infty;$

\item[($h_3$)] $V(x)$ is $\tau$-periodic in $x\in\mathbb{Z}^3$ with $\tau\in\mathbb{Z};$

\item[($f_1$)] $f(t)$ is continuous in $t\in\mathbb{R}$ and $f(t)=o(t)$ as $|t|\rightarrow 0$;
\item[($f_2$)] there are constants $c>0$ and $p>\frac{3+\alpha}{3}$ such that
$$
|f(t)|\leq c(1+|t|^{p-1}), \quad t\in \mathbb{R};
$$ 
\item[($f_3$)] there exists $\theta>4$ such that $$0\leq \theta F(t)=\theta \int_{0}^{t}f(s)\,ds \leq 2f(t)t,\quad t\in \mathbb{R};$$
\item[($f_4$)] for any $u\in H\backslash\{0\}$,$$\frac{\int_{\mathbb{Z}^3}(R_\alpha \ast F(tu))f(tu)u\,d\mu}{t^3}$$ is strictly increasing with respect $t\in (0, \infty)$.
\end{itemize}

 By $(f_1)$ and $(f_2)$, we have that for any $\varepsilon>0$, there exists $C_\varepsilon>0$ such that
\begin{equation}\label{ad}
|f(t)|\leq\varepsilon|t|+C_\varepsilon|t|^{p-1},\quad t\in \mathbb{R}.
\end{equation}
Hence
\begin{equation}\label{ae}
|F(t)|\leq\varepsilon|t|^2+C_\varepsilon|t|^{p},\quad t\in \mathbb{R}.
\end{equation}


Let $H^{1}(\mathbb{Z}^3)$ be the completion of $C_c(\mathbb{Z}^3)$ with respect to the norm
\begin{eqnarray*}
\|u\|_{H^{1}}=\left(\int_{\mathbb{Z}^3}(|\nabla u|^2+u^2)\,d\mu\right)^{\frac{1}{2}}.
\end{eqnarray*}
Let $V(x) \geqslant V_{0}>0$. We introduce a new subspace:
$$
H=\left\{u \in H^{1}\left(\mathbb{Z}^{3}\right): \int_{\mathbb{Z}^{3}} V(x) u^{2}\,d \mu<\infty\right\}
$$
with a norm
$$
\|u\|=\left(\int_{\mathbb{Z}^{3}}\left(a|\nabla u|^{2}+V(x) u^{2}\right) d \mu\right)^{\frac{1}{2}},
$$
where $a$ is a positive constant. The space $H$ is a Hilbert space and its inner product is
$$
(u, v)=\int_{\mathbb{Z}^3}\left(a\nabla u \nabla v+V u v\right)\, d \mu.
$$
Clearly, for any $u \in H$ and $q\geq 2$, we have
\begin{equation}\label{ac}
\|u\|_{q} \leq \|u\|_{2}\leq C \|u\|.
\end{equation}

The energy functional $J(u): H\rightarrow\R$ associated to the equation (\ref{aa}) is given by
$$
J(u)=\frac{1}{2} \int_{\mathbb{Z}^{3}}\left(a|\nabla u|^{2}+V(x) u^{2}\right) d \mu+\frac{b}{4}\left(\int_{\mathbb{Z}^{3}}|\nabla u|^{2} d \mu\right)^{2}-\frac{1}{2} \int_{\mathbb{Z}^3}(R_\alpha \ast F(u))F(u)\,d\mu.
$$
Moreover, for any $\phi\in H$, one gets easily that
$$
\left\langle J^{\prime}(u), \phi\right\rangle=\int_{\mathbb{Z}^{3}}(a \nabla u \nabla \phi+V(x) u \phi) d \mu+b \int_{\mathbb{Z}^{3}}|\nabla u|^{2} d \mu \int_{\mathbb{Z}^{3}} \nabla u \nabla \phi d \mu-\int_{\mathbb{Z}^3}(R_\alpha \ast F(u))f(u)\phi\,d\mu.
$$
We say that $u\in H$ is a nontrivial solution of the equation (\ref{aa}), if $u$ is a nonzero critical point of $J$, i.e. $J'(u)=0$ with $u\neq 0$. A ground state solution of the equation (\ref{aa}) means that
$u$ is a nonzero critical point of $J$ with the least energy, that is,
$$
J(u)=\inf\limits_{\mathcal{N}} J>0,
$$
where 
$$
\mathcal{N}=\left\{u \in H \backslash\{0\}:\left\langle J^{\prime}(u), u\right\rangle=0\right\}
$$
is the Nehari manifold.

Now we state our main results.

\begin{thm}\label{t0}
Let $(h_1)$, $(h_2)$ and $(f_1)$-$(f_3)$ hold. Then the equation (\ref{aa}) has a nontrivial solution. 
\end{thm}

\begin{thm}\label{t1}
Let $(h_1)$, $(h_2)$ and $(f_1)$-$(f_4)$ hold. Then the equation (\ref{aa}) has a ground state solution.    
\end{thm}
\begin{thm}\label{t2}
 Let $(h_1)$, $(h_3)$ and $(f_1)$-$(f_4)$ hold. Then the equation (\ref{aa}) has a ground state solution.
    
\end{thm}

The rest of this paper is organized as follows. In Section 2, we present some preliminary results on graphs. In Section 3, we prove Theorem \ref{t0} by mountain pass theorem. In Section 4, we prove Theorem \ref{t1} based on the mountain pass theorem and Nehari manifold approach. In Section 5, we prove Theorem \ref{t2} by the method of generalized Nehari manifold.

\section{Preliminaries} 
In this section, we introduce the basic settings on graphs and give some basic results. 

Let $G=(\mathbb{V}, \mathbb{E})$ be a connected, locally finite graph, where $\mathbb{V}$ denotes the vertex set and $\mathbb{E}$ denotes the edge set. We call vertices $x$ and $y$ neighbors, denoted by $x \sim y$, if there exists an edge connecting them, i.e. $(x, y) \in \mathbb{E}$. For any $x,y\in \mathbb{V}$, the distance $d(x,y)$ is defined as the minimum number of edges connecting $x$ and $y$, i.e.
$$d(x,y)=\inf\{k:x=x_0\sim\cdots\sim x_k=y\}.$$
Let $B_{r}(a)=\{x\in\mathbb{V}: d(x,a)\leq r\}$ be the closed ball of radius $r$ centered at $a\in \mathbb{V}$. For brevity, we write $B_{r}:=B_{r}(0)$.

In this paper, we consider, the natural discrete model of Euclidean space, the integer lattice graph.  The $3$-dimensional integer lattice graph, denoted by $\mathbb{Z}^3$, consists of the set of vertices $\mathbb{V}=\mathbb{Z}^3$ and the set of edges $\mathbb{E}=\{(x,y): x,\,y\in\mathbb{Z}^{3},\,\underset {{i=1}}{\overset{3}{\sum}}|x_{i}-y_{i}|=1\}.$
Here we denote $|x-y|:=d(x,y)$ on the lattice graph $\mathbb{Z}^{3}$.

For $u,v \in C\left(\mathbb{Z}^{3}\right)$, we define the Laplacian of $u$ as
$$
\Delta u(x)=\sum\limits_{y \sim x}(u(y)-u(x)),
$$
 and the gradient form $\Gamma$ is defined as
$$
\Gamma(u, v)(x)=\frac{1}{2} \sum\limits_{y \sim x}(u(y)-u(x))(v(y)-v(x)) .
$$
We write $\Gamma(u)=\Gamma(u, u)$ and denote the length of the gradient as
$$
|\nabla u|(x)=\sqrt{\Gamma(u)(x)}=\left(\frac{1}{2} \sum\limits_{y \sim x}(u(y)-u(x))^{2}\right)^{\frac{1}{2}}.
$$

The space $\ell^{p}\left(\mathbb{Z}^{3}\right)$ is defined as

$$
\ell^{p}\left(\mathbb{Z}^{3}\right)=\left\{u \in C\left(\mathbb{Z}^{3}\right):\|u\|_{p}<\infty\right\},
$$

where

$$
\|u\|_{p}= \begin{cases}\left(\sum\limits_{x \in \mathbb{Z}^{3}}|u(x)|^{p}\right)^{\frac{1}{p}}, & \text { if } 1 \leq p<\infty, \\ \sup\limits_{x \in \mathbb{Z}^{3}}|u(x)|, & \text { if } p=\infty.\end{cases}
$$

The following discrete Hardy-Littlewood-Sobolev (HLS for abbreviation) inequality plays a key role in this paper, see \cite{LW,W1}.

\begin{lm}\label{lm1}
Let $0<\alpha <3,\,1<r,s<\infty$ and $\frac{1}{r}+\frac{1}{s}+\frac{3-\alpha}{3}=2$. We have the discrete
HLS inequality
\begin{equation}\label{bo}
\int_{\mathbb{Z}^3}(R_\alpha\ast u)(x)v(x)\,d\mu\leq C_{r,s,\alpha}\|u\|_r\|v\|_s,\quad u\in \ell^r(\mathbb{Z}^3),\,v\in \ell^s(\mathbb{Z}^3).
\end{equation}
And an equivalent form is
\begin{equation}\label{p1}
\|R_\alpha\ast u\|_{\frac{3r}{3-\alpha r}}\leq C_{r,\alpha}\|u\|_r,\quad u\in \ell^r(\mathbb{Z}^3),
\end{equation}
where $0<\alpha <3,\,1<r<\frac{3}{\alpha}$.
\end{lm}

Denote
$$I(u):=\frac{1}{2} \int_{\mathbb{Z}^3}(R_\alpha \ast F(u))F(u)\,d\mu, \quad u\in H.$$
Then for any $\phi\in H,$ we have
$$\langle I'(u), \phi\rangle=\int_{\mathbb{Z}^3}(R_\alpha \ast F(u))f(u)\phi\,d\mu.$$

\begin{lm}\label{lhj}
Let $(f_1)$-$(f_3)$ hold. Then 
\begin{itemize}
    \item[(i)] $I$ is weakly lower semicontinuous;
    \item[(ii)]  $I'$ is weakly continuous. 
\end{itemize}    
\end{lm}
\begin{proof}
 Let $u_n\rightharpoonup u$ in $H$. Then $\{u_n\}$ is bounded in $H$, and hence bounded in $\ell^\infty(\mathbb{Z}^3)$. Therefore, by diagonal principle, there exists a subsequence of $\{u_n\}$ (still denoted by itself) such that  \begin{equation}\label{jm}
  u_n\rightarrow u,\quad \text{pointwise~in}~\mathbb{Z}^3.   
 \end{equation}
 
(i) By Fatou's lemma, we get that
 \begin{eqnarray*}
 I(u)=\frac{1}{2} \int_{\mathbb{Z}^3}(R_\alpha \ast F(u))F(u)\,d\mu\leq  \liminf_{n \rightarrow \infty}\frac{1}{2} \int_{\mathbb{Z}^3}(R_\alpha \ast F(u_n))F(u_n)\,d\mu= \liminf_{n \rightarrow \infty} I(u_n), 
 \end{eqnarray*}
 which implies that $I$ is weakly lower semicontinuous.
 
 \
\

(ii) Since $C_c(\mathbb{Z}^3)$ is dense in $H$, we only need to show that for any $\phi\in C_c(\mathbb{Z}^3)$,
 \begin{equation}\label{jn}
  \langle I'(u_n)-I'(u),\phi\rangle\rightarrow 0,\quad n\rightarrow\infty.   
 \end{equation}
 In fact, assume $\text{supp}(\phi)\subset B_r$ with $r>1$. Direct calculation yields that
 \begin{eqnarray*}
  \langle I'(u_n)-I'(u),\phi\rangle&=&\int_{\mathbb{Z}^3}(R_\alpha \ast \left(F(u_n)-F(u)\right)f(u)\phi\,d\mu\\&&+\int_{\mathbb{Z}^3}(R_\alpha \ast F(u_n))\left[f(u_n)-f(u)\right]\phi\,d\mu\\&=&T_1+T_2.   
 \end{eqnarray*}
By (\ref{ae}) and (\ref{ac}), one gets easily that $\{F(u_n)\}$ is bounded in $\ell^{\frac{6}{3+\alpha}}(\mathbb{Z}^3)$. Then it follows from the HLS inequality (\ref{p1}) that $\{(R_\alpha \ast F(u_n))\}$ is bounded in $\ell^{\frac{6}{3-\alpha}}(\mathbb{Z}^3)$. Moreover, we have $F(u_n)\rightarrow F(u)$ pointwise in $\mathbb{Z}^3$. Passing to a subsequence, we have
$$(R_\alpha \ast F(u_n))\rightharpoonup (R_\alpha \ast F(u)), \quad\text{in~}\ell^{\frac{6}{3-\alpha}}(\mathbb{Z}^3).$$ 
Since $f(u)\phi\in\ell^{\frac{6}{3+\alpha}}(\mathbb{Z}^3) $, we get that $$T_1\rightarrow 0,\quad n\rightarrow\infty.$$
By the HLS inequality (\ref{bo}) and (\ref{jm}), we obtain that
\begin{eqnarray*}
    T_2&\leq& C\|F(u_n)\|_{\frac{6}{3+\alpha}}\left(\int_{\mathbb{Z}^3}|(f(u_n)-f(u))\phi|^{\frac{6}{3+\alpha}}\,d\mu|\right)^{\frac{3+\alpha}{6}}\\&\leq& C\left(\int_{\mathbb{Z}^3}|(f(u_n)-f(u))\phi|^{\frac{6}{3+\alpha}}\,d\mu|\right)^{\frac{3+\alpha}{6}}\\&=&C\left(\int_{B_r}|(f(u_n)-f(u))\phi|^{\frac{6}{3+\alpha}}\,d\mu\right)^{\frac{3+\alpha}{6}}\\&\rightarrow& 0,\quad n\rightarrow\infty.
\end{eqnarray*}
Then (\ref{jn}) follows from $T_1,T_2\rightarrow 0$ as $n\rightarrow\infty.$
\end{proof}

For any $u\in H\backslash\{0\}$, let $$g(t):=I(tu)=\frac{1}{2}\int_{\mathbb{Z}^3}\left(R_\alpha \ast F\left(tu\right)\right)F\left(tu\right)\,d\mu,\quad t\geq 0.$$

\begin{lm}\label{lhf}
Let $(f_1)$-$(f_4)$ hold. Then
\begin{itemize}
    \item[(i)] for $t>0$, $\left[\frac{1}{4}t g'(t)-g(t)\right]$ is a positive and strictly increasing function;
    \item[(ii)] for $t\geq 1$, we have $g(t)\geq t^\theta g(1).$
\end{itemize}
\end{lm}
\begin{proof}
 (i) For $t>0$, by $(f_3)$, we get that
\begin{eqnarray*}
 g'(t)&=&\langle I'(tu),u\rangle=\int_{\mathbb{Z}^3}\left(R_\alpha \ast F(tu)\right)f(tu)u\,d\mu \\&=&\frac{1}{t}\int_{\mathbb{Z}^3}\left(R_\alpha \ast F(tu)\right)f(tu)tu\,d\mu\\& \geq& \frac{\theta}{2t}\int_{\mathbb{Z}^3}\left(R_\alpha \ast F(tu)\right)F(tu)\,d\mu\\&=& \frac{\theta}{t} g(t)\\&>& \frac{4}{t} g(t),
\end{eqnarray*} 
which implies that $\frac{1}{4}t g'(t)-g(t)>0.$

 By $(f_4)$, one gets that $\frac{g'(t)}{t^3}$ is strictly increasing for  $t>0.$ Hence
\begin{eqnarray*}
    \frac{1}{4}t g'(t)-g(t)=\int_0^t\left(\frac{g'(t)}{t^3}-\frac{g'(s)}{s^3}\right)s^3\,ds
\end{eqnarray*}
strictly increases for $t>0$.

\
\

(ii) 
Clearly for $t=1$, the result holds.
From the proof of (i), one gets that $$g'(s)\geq \frac{\theta}{s}g(s),\quad s>0.$$
As a consequence, for $t>1$,
$$\int_1^{t} \frac{dg}{g}\geq \theta \int_1^{t} \frac{ds}{s},$$
which implies that
$$g(t)\geq  t^\theta g(1).$$

\end{proof}

Finally, we state some results about the compactness of $H$. The following one can be seen in \cite{ZZ}

\begin{lm}\label{lgg}
Let $(h_1)$ and $(h_2)$ hold. Then for any $q\geq 2$, $H$ is compactly embedded into $\ell^{q}\left(\mathbb{Z}^{3}\right)$. That is, there exists a constant $C$ such that, for any $u \in H$,
$$
\|u\|_{q} \leq C_q\|u\|.
$$
Furthermore, for any bounded sequence $\left\{u_{n}\right\} \subset H$, there exists $u \in H$ such that, up to a subsequence, 
$$
\begin{cases}u_{n} \rightharpoonup u, & \text { in } H, \\ u_{n}(x) \rightarrow u(x), & \text{pointwise~in}~ \mathbb{Z}^{3}, \\ u_{n} \rightarrow u, & \text { in } \ell^{q}\left(\mathbb{Z}^{3}\right) .\end{cases}
$$ 
\end{lm}

We also present a discrete Lions lemma, which
denies a sequence $\left\{u_{n}\right\}$ to distribute itself over $\mathbb{Z}^3.$

\begin{lm}\label{lgh}
 Let $2\leq s<\infty$. Assume that $\left\{u_{n}\right\}$ is bounded in $H$ and

$$
\left\|u_{n}\right\|_{\infty} \rightarrow 0,\quad n \rightarrow\infty \text {. }
$$
Then, for any $s<t<\infty$,
$$
u_{n} \rightarrow 0,\quad  \text { in } \ell^{t}\left(\mathbb{Z}^{3}\right) \text {. }
$$   
\end{lm} 

\begin{proof}
 By (\ref{ac}), we get that $\{u_n\}$ is bounded in $\ell^{s}\left(\mathbb{Z}^{3}\right)$. Hence, for $s<t<x\infty$, this result follows from an interpolation inequality
$$
\left\|u_{n}\right\|_{t}^{t} \leq\left\|u_{n}\right\|_{s}^{s}\left\|u_{n}\right\|_{\infty}^{t-s} .
$$
 
\end{proof}

\section{Proof of Theorem \ref{t0}}
In this section, we prove the existence of nontrivial solutions to the equation (\ref{aa}) by the mountain pass theorem. First we show that
 the functional $J(u)$ satisfies the mountain pass geometry.
\begin{lm}\label{lm}
  Let $(h_1)$ and $(f_1)$-$(f_3)$ hold. Then
\begin{itemize}
    \item[(i)] there exist $\sigma, \rho>0$ such that $J(u) \geq \sigma>0$ for all $\|u\|=\rho$;
    \item[(ii)] there exists $e \in H$ with $\|e\|>\rho$ such that $J(e)< 0$.  
\end{itemize}  
\end{lm}
\begin{proof}
  (i)
By (\ref{ae}) and the HLS inequality (\ref{bo}), we get that
\begin{eqnarray}\label{uc}
    \int_{\mathbb{Z}^3}(R_\alpha \ast F(u))F(u)\,d\mu\nonumber&\leq& C\left(\int_{\mathbb{Z}^3} |F(u_n)|^{\frac{6}{3+\alpha}}\,d\mu  \right)^{\frac{3+\alpha}{3}}\\&\leq &C\left(\int_{\mathbb{Z}^3} \left(\varepsilon|u|^2+C_\varepsilon|u|^{p}\right)^{\frac{6}{3+\alpha}}\,d\mu  \right)^{\frac{3+\alpha}{3}}\nonumber\\&\leq& \varepsilon\|u\|^4_{\frac{12}{3+\alpha}}+C_\varepsilon\|u\|^{2p}_{\frac{6p}{3+\alpha}}\nonumber\\&\leq& \varepsilon\|u\|^4+C_\varepsilon\|u\|^{2p}.
\end{eqnarray}
Then 
\begin{align*}
J(u) & =\frac{1}{2}\|u\|^2+\frac{b}{4}\left(\int_{\mathbb{Z}^3}|\nabla u|^{2} d \mu\right)^{2}-\frac{1}{2}  \int_{\mathbb{Z}^3}(R_\alpha \ast F(u))F(u)\,d\mu\\
& \geq \frac{1}{2}\|u\|^2-\frac{1}{2}  \int_{\mathbb{Z}^3}(R_\alpha \ast F(u))F(u)\,d\mu \\
& \geq \frac{1}{2}\|u\|^2-\varepsilon\|u\|^4-C_\varepsilon\|u\|^{2p}.
\end{align*}
Note that $p>\frac{3+\alpha}{3}>1$. Let $\varepsilon\rightarrow 0^+$,  then there exist $\sigma, \rho>0$ small enough such that $J(u) \geq \sigma>0$ for $\|u\|=\rho$.

\
\

(ii) Let $u\in H\backslash\{0\}$. Then if follows from Lemma \ref{lhf} (ii), (\ref{uc}) and $\theta>4$ that
\begin{eqnarray}\label{ud}
 \lim _{t \rightarrow\infty} J(t u)\nonumber&=&\lim _{t \rightarrow\infty}\left[\frac{t^{2}}{2}\|u\|^2+\frac{b t^{4}}{4}\left(\int_{\mathbb{Z}^3}|\nabla u|^{2} d \mu\right)^{2}-\frac{1}{2} \int_{\mathbb{Z}^3}(R_\alpha \ast F(tu))F(tu)\,d\mu\right]\nonumber\\&\leq&\lim _{t \rightarrow\infty}\left[\frac{t^{2}}{2}\|u\|^2+\frac{b t^{4}}{4}\left(\int_{\mathbb{Z}^3}|\nabla u|^{2} d \mu\right)^{2}-\frac{t^\theta}{2} \int_{\mathbb{Z}^3}\left(R_\alpha \ast F(u)\right)F(u)\,d\mu\right]\nonumber\\&\rightarrow&-\infty.   
\end{eqnarray}
Hence, we can choose $t_{0}>0$ large enough such that $\left\|e\right\|>\rho$ with $e=t_{0} u$ and $J\left(e\right)<0$. 
\end{proof}

In the following, we prove the compactness of Palais-Smale sequence. Recall that, for a given functional $\Phi\in C^{1}(X,\mathbb{R})$, where $X$ is a Banach space, a sequence $\{u_n\}\subset X$ is a Palais-Smale sequence at level $c\in\mathbb{R}$, $(PS)_c$ sequence for short, of the functional $\Phi$, if it satisfies, as $n\rightarrow\infty$,
\begin{eqnarray*}
\Phi(u_n)\rightarrow c, \qquad \text{in}~ X,\qquad\text{and}\qquad
\Phi'(u_n)\rightarrow 0, \qquad \text{in}~X^*
\end{eqnarray*}
where $X^{*}$ is the dual space of $X$. Moreover, we say that $\Phi$ satisfies $(PS)_c$ condition, if any $(PS)_c$ sequence has a convergent subsequence. 

\begin{lm}\label{ln}
Let $(h_1),(h_2)$ and $(f_1)$-$(f_3)$ hold. Then $J$ satisfies the $(PS)_c$ condition with $c\in\mathbb{R}.$ 
\end{lm}
\begin{proof}
For any $c\in\mathbb{R}$, let $\left\{u_{n}\right\}$ be a $(P S)_{c}$ sequence for $J(u)$, namely
\begin{equation}\label{ag}
 J\left(u_{n}\right)=c+o_{n}(1), \quad \text { and } \quad J^{\prime}\left(u_{n}\right)=o_{n}(1),
 \end{equation}
 where $o_n(1)\rightarrow 0$ as $n\rightarrow\infty.$
 
Note that $\theta>4$ and $b>0$. By (\ref{ag}), we get that
\begin{eqnarray}\label{ba}
\|u_n\|^2\nonumber&=&\int_{\mathbb{Z}^3}(R_\alpha \ast F(u_n))F(u_n)\,d\mu-\frac{b}{2}\left(\int_{\mathbb{Z}^3}\left|\nabla u_{n}\right|^{2} d \mu\right)^{2}+2c+o_n(1)\nonumber\\&\leq& \frac{2}{\theta}\int_{\mathbb{Z}^3}(R_\alpha \ast F(u_n))f(u_n)u_n\,d\mu-\frac{b}{2}\left(\int_{\mathbb{Z}^3}\left|\nabla u_{n}\right|^{2} d \mu\right)^{2}+2c+o_n(1)\nonumber\\&\leq&\frac{1}{2} \left(\|u_n\|^2+b\left(\int_{\mathbb{Z}^3}\left|\nabla u_{n}\right|^{2} d \mu\right)^{2}+o_n(1)\|u_n\|\right)-\frac{b}{2}\left(\int_{\mathbb{Z}^3}\left|\nabla u_{n}\right|^{2} d \mu\right)^{2}+2c+o_n(1)\nonumber\\&=&\frac{1}{2}\|u_n\|^2+o_n(1)\|u_n\|+2c+o_n(1),
\end{eqnarray}
which implies that $\{u_n\}$ is bounded in $H$. Then by Lemma \ref{lgg}, up to a subsequence, there exists $u \in H$ such that
\begin{equation}\label{ua}
\begin{cases}u_{n} \rightharpoonup u, & \text { in } H, \\ u_{n}(x) \rightarrow u(x), & \text{pointwise~in~} \mathbb{Z}^{3}, \\ u_{n} \rightarrow u, & \text { in } \ell^{q}\left(\mathbb{Z}^{3}\right),q\geq 2.\end{cases}
\end{equation}
Since $|\nabla u(x)|^2=\frac{1}{2}\underset {y\sim x}{\sum}(u(y)-u(x))^{2}$, one gets easily that$$\int_{\mathbb{Z}^3}|\nabla u|^2\,d\mu\leq C\| u \|^2_2.$$
Hence by H\"{o}lder inequality, the boundedness of $\{u_n\}$ and (\ref{ua}), we have
\begin{eqnarray}\label{ak}
    \int_{\mathbb{Z}^3} |\nabla u_n| |\nabla (u_n-u)|\,d\mu\nonumber&\leq&\left(\int_{\mathbb{Z}^3} |\nabla u_n|^2\,d\mu\right)^{\frac{1}{2}}\left(\int_{\mathbb{Z}^3} |\nabla (u_n-u)|^2\,d\mu\right)^{\frac{1}{2}}\nonumber\\&\leq&C\|u_n\|\|u_n-u\|_2\nonumber\\&\rightarrow& 0,\quad n\rightarrow\infty.
\end{eqnarray}
Moreover, by the HLS inequality (\ref{bo}), H\"{o}lder inequality, the boundedness of $\{u_n\}$ and (\ref{ua}), we have
\begin{eqnarray}\label{ub}
|\langle I'(u_n),u_n-u\rangle|\nonumber&\leq&\int_{\mathbb{Z}^3}(R_\alpha \ast F(u_n))|f(u_n)(u_n-u)|\nonumber\,d\mu
\\&\leq& C\left(\int_{\mathbb{Z}^3} |F(u_n)|^{\frac{6}{3+\alpha}}\,d\mu  \right)^{\frac{3+\alpha}{6}}\left(\int_{\mathbb{Z}^3} |f(u_n)(u_n-u)|^{\frac{6}{3+\alpha}}\,d\mu  \right)^{\frac{3+\alpha}{6}}\nonumber\\&\leq&C\left(\int_{\mathbb{Z}^3} (|u_n|^2+|u_n|^p)^{\frac{6}{3+\alpha}}\,d\mu  \right)^{\frac{3+\alpha}{6}}\left(\int_{\mathbb{Z}^3} \left[(|u_n|+|u_n|^{p-1})|u_n-u|\right]^{\frac{6}{3+\alpha}}\,d\mu  \right)^{\frac{3+\alpha}{6}}\nonumber\\&\leq& C\left[\left(\int_{\mathbb{Z}^3} (|u_n||u_n-u|)^{\frac{6}{3+\alpha}}\,d\mu  \right)^{\frac{3+\alpha}{6}}+\left(\int_{\mathbb{Z}^3} (|u_n|^{p-1}|u_n-u|)^{\frac{6}{3+\alpha}}\,d\mu  \right)^{\frac{3+\alpha}{6}}\right]\nonumber\\&\leq& C\|u_n\|_{\frac{12}{3+\alpha}}\|u_n-u\|_{\frac{12}{3+\alpha}}+C\|u_n\|^{p-1}_{\frac{6p}{3+\alpha}}\|u_n-u\|_{\frac{6p}{3+\alpha}}\nonumber\\&\leq& C\|u_n-u\|_{\frac{12}{3+\alpha}}+C\|u_n-u\|_{\frac{6p}{3+\alpha}}\nonumber\\&\rightarrow&0,\quad n\rightarrow\infty.
\end{eqnarray}
Then if follows from (\ref{ag}), (\ref{ak}) and (\ref{ub}) that
\begin{eqnarray*}
    |(u_n,u_n-u)|&\leq& |\langle J'(u_n),u_n-u\rangle|+b\int_{\mathbb{Z}^3}|\nabla u_n|^{2} d \mu\int_{\mathbb{Z}^3}|\nabla u_n||\nabla (u_n-u)|\,d \mu+|\langle I'(u_n),u_n-u\rangle|\nonumber\\&\leq&o_n(1)\|u_n-u\|+Cb\|u_n\|^2\int_{\mathbb{Z}^3}|\nabla u_n||\nabla (u_n-u)|\,d \mu+|\langle I'(u_n),u_n-u\rangle|\nonumber\\&\rightarrow& 0,\quad n\rightarrow\infty.
\end{eqnarray*}
Furthermore, since $u_n\rightharpoonup u$ in $H$, we have
$$(u,u_n-u)\rightarrow 0,\quad n\rightarrow\infty.$$
Hence 
$$\|u_n-u\|\rightarrow 0,\quad n\rightarrow\infty.$$
Note that $u_n\rightarrow u$ pointwise in $\mathbb{Z}^3$, we obtain $u_n\rightarrow u$ in $H.$

\end{proof}

{\bf Proof of Theorem \ref{t0} :} By Lemma \ref{lm}, one sees that $J$ satisfies the geometric structure of the
mountain pass theorem. Hence there exists a $(PS)_c$ sequence with $c=\inf\limits_{\gamma\in\Gamma}\max\limits_{t\in [0,1]} J(\gamma(t))$, 
where 
$$\Gamma=\{\gamma\in C([0,1],H):\gamma(0)=0, \gamma(1)=e\}.$$ By Lemma \ref{ln}, $J$ satisfies the $(PS)_c$ condition. Then $c$ is a critical value of $J$ by the mountain pass theorem due to Ambrosetti-Rabinowitz \cite{W0}. In particular, there exists $u\in H$ such that $J(u)=c$. Since $J(u)=c\geq\sigma>0$, we have $u\neq 0$. Hence the equation (\ref{aa}) possesses at least a nontrivial solution.\qed

\section{Proof of Theorem \ref{t1}}

In this section, we prove the existence of ground state solutions to the equation  (\ref{aa}) under the  conditions $(h_1)$ and $(h_2)$ on $V$. Now we show some properties of $J$ on $\mathcal{N}$ that are useful in our proofs.

\begin{lm}\label{lgj}
 Let $\left(h_{1}\right)$ and $\left(f_{1}\right)$-$\left(f_{4}\right)$ hold. Then 
\begin{itemize}
    \item[(i)] for any $u \in H \backslash\{0\}$, there exists a unique $s_{u}>0$ such that $s_{u} u \in \mathcal{N}$ and $J(s_{u} u)=$ $\max\limits_{s>0} J(s u)$;
    
    \item[(ii)] there exists $\eta>0$ such that $\|u\| \geq \eta$ for  $u \in \mathcal{N}$;

    \item[(iii)] $J$ is bounded from below on $\mathcal{N}$ by a positive constant.


\end{itemize}
  
\end{lm}

\begin{proof}(i) For any $u \in H \backslash\{0\}$ and $s>0$, similar to $(\ref{uc})$, we get that
\begin{eqnarray*}\label{af}
    \int_{\mathbb{Z}^3}(R_\alpha \ast F(su))F(su)\,d\mu \leq\varepsilon s^4\|u\|^4+C_\varepsilon s^{2p}\|u\|^{2p}.
\end{eqnarray*}
Then
\begin{equation}\label{jv}
\begin{aligned}
J(s u) & =\frac{s^{2}}{2} \int_{\mathbb{Z}^{3}}\left(a|\nabla u|^{2}+V(x) u^{2}\right) d \mu+\frac{b s^{4}}{4}\left(\int_{\mathbb{Z}^{3}}|\nabla u|^{2} d \mu\right)^{2}-\frac{1}{2}\int_{\mathbb{Z}^3}(R_\alpha \ast F(su))F(su)\,d\mu  \\
& =\frac{s^{2}}{2}\|u\|^{2}+\frac{b s^{4}}{4}\left(\int_{\mathbb{Z}^{3}}|\nabla u|^{2} d \mu\right)^{2}-\frac{1}{2}\int_{\mathbb{Z}^3}(R_\alpha \ast F(su))F(su)\,d\mu \\&\geq\frac{s^{2}}{2}\|u\|^{2}-\varepsilon s^4\|u\|^4-C_\varepsilon s^{2p}\|u\|^{2p}.
\end{aligned}
\end{equation}
Since $p>\frac{3+\alpha}{3}>1$, let $\varepsilon\rightarrow 0^+$, we get easily that $J(s u)>0$ for $s>0$ small enough.

On the other hand, similar to (\ref{ud}), we get that
$$
J(s u) \rightarrow-\infty, \quad s \rightarrow \infty .
$$
Therefore, $\max\limits_{s>0} J(s u)$ is achieved at some $s_{u}>0$ with $s_{u} u \in \mathcal{N}$. In the following, we show the uniqueness of $s_{u}$ by a contradiction. Suppose that there exist $s_{u}^{\prime}>s_{u}>0$ such that $s_{u}^{\prime} u, s_{u} u \in \mathcal{N}$, then
$$
\begin{aligned}
& \frac{1}{\left(s_{u}^{\prime}\right)^{2}}\|u\|^2+b\left(\int_{\mathbb{Z}^{3}}|\nabla u|^{2} d \mu\right)^{2}=\int_{\mathbb{Z}^3}\frac{(R_\alpha \ast F(s'_uu))f(s'_uu)u}{\left(s_{u}^{\prime}\right)^{3}}\,d\mu, \\
& \frac{1}{\left(s_{u}\right)^{2}}\|u\|^2+b\left(\int_{\mathbb{Z}^{3}}|\nabla u|^{2} d \mu\right)^{2}=\int_{\mathbb{Z}^3}\frac{(R_\alpha \ast F(s_uu))f(s_uu)u}{\left(s_{u}\right)^{3}}\,d\mu.
\end{aligned}
$$
As a consequence,
$$
\begin{aligned}
\left(\frac{1}{\left(s_{u}^{\prime}\right)^{2}}-\frac{1}{\left(s_{u}\right)^{2}}\right)\|u\|^2=\int_{\mathbb{Z}^3}\frac{(R_\alpha \ast F(s'_uu))f(s'_uu)u}{\left(s_{u}^{\prime}\right)^{3}}\,d\mu-\int_{\mathbb{Z}^3}\frac{(R_\alpha \ast F(s_uu))f(s_uu)u}{\left(s_{u}\right)^{3}}\,d\mu,
\end{aligned}
$$
which is contradiction in view of $\left(f_{4}\right)$ and $s_{u}^{\prime}>s_{u}>0$.

\
\

(ii) By the HLS inequality (\ref{bo}), we have
\begin{eqnarray}\label{aj}
\int_{\mathbb{Z}^3}(R_\alpha \ast F(u))f(u)u\nonumber\,d\mu
&\leq& C\left(\int_{\mathbb{Z}^3} |F(u)|^{\frac{6}{3+\alpha}}\,d\mu  \right)^{\frac{3+\alpha}{6}}\left(\int_{\mathbb{Z}^3} |f(u)u|^{\frac{6}{3+\alpha}}\,d\mu  \right)^{\frac{3+\alpha}{6}}\nonumber\\&\leq&C\left(\int_{\mathbb{Z}^3} \left(\varepsilon|u|^2+C_\varepsilon|u|^{p}\right)^{\frac{6}{3+\alpha}}\,d\mu  \right)^{\frac{3+\alpha}{3}}\nonumber\\&\leq& \varepsilon\|u\|^4_{\frac{12}{3+\alpha}}+C_\varepsilon\|u\|^{2p}_{\frac{6p}{3+\alpha}}\nonumber\\&\leq& \varepsilon\|u\|^4+C_\varepsilon\|u\|^{2p}.
\end{eqnarray}

Let $u \in\mathcal{N}$. Then
$$
\begin{aligned}
0&=\left\langle J^{\prime}(u), u\right\rangle\\ & =\|u\|^{2}+b\left(\int_{\mathbb{Z}^{3}}|\nabla u|^{2} d \mu\right)^{2}-\int_{\mathbb{Z}^3}(R_\alpha \ast F(u))f(u)u\,d\mu  \\
& \geq\|u\|^{2}-\varepsilon\|u\|^{4}-C_{\varepsilon}\|u\|^{2p} .
\end{aligned}
$$
Since $p>1$, we get easily that there exists a constant $\eta>0$ such that $\|u\| \geq \eta>0$ with $u \in \mathcal{N}$.

\
\

(iii) For any $u \in \mathcal{N}$, by $(f_3)$ and (ii), we derive that
\begin{eqnarray*}
J(u) &=& J(u)-\frac{1}{\theta}\left\langle J^{\prime}(u), u\right\rangle \\
&=&\left(\frac{1}{2}-\frac{1}{\theta}\right)\|u\|^{2}+b\left(\frac{1}{4}-\frac{1}{\theta}\right)\left(\int_{\mathbb{Z}^{3}}|\nabla u|^{2} d \mu\right)^{2}\\&&+\frac{1}{2}\int_{\mathbb{Z}^3}(R_\alpha \ast F(u))\left(\frac{2}{\theta}f(u)u-F(u)\right)\,d\mu\\
& \geq &\left(\frac{1}{2}-\frac{1}{\theta}\right)\|u\|^{2} \\& \geq &\left(\frac{1}{2}-\frac{1}{\theta}\right) \eta^{2}\\&>&0.
\end{eqnarray*}

\end{proof}

In the following, we establish a  homeomorphic map between a unit sphere $S\subset H$ and $\mathcal{N}$. 

\begin{lm}\label{lg}
  Let $\left(h_{1}\right)$ and $\left(f_{1}\right)$-$\left(f_{4}\right)$ hold. Define the maps $s:H\backslash\{0\}\rightarrow (0,\infty)$, $u\mapsto s_u$ and $$
\begin{aligned}
\widehat{m}: H\backslash\{0\} & \rightarrow \mathcal{N}, \\
u & \mapsto \widehat{m}(u)=s_{u} u.
\end{aligned}
$$ 
Then 
\begin{itemize}
    \item [(i)] the maps $s$ and $\widehat{m}$ are continuous;
    \item[(ii)] the map $m:=\widehat{m}\mid_{S}$ is a homeomorphism between $S$ and $\mathcal{N}$, and the inverse of $m$ is given by
\begin{equation*}
m^{-1}(u)=\frac{u}{\|u\|}. 
\end{equation*}
\end{itemize}
\end{lm}
\begin{proof}
(i) Let $u_n\rightarrow u$ in $H\backslash\{0\}$. Denote $s_n=s_{u_n}$, then $\widehat{m}(u_n)=s_n u_n\in\mathcal{N}$. Since, for any $s>0$, $\widehat{m}(su)=\widehat{m}(u)$, without loss of generality, we may assume that $\{u_n\}\subset S$. By Lemma \ref{lgj}(ii), we get that $$s_n=\|s_nu_n\|\geq\eta>0.$$ 
    
     We claim that $\{s_n\}$ is bounded. Otherwise, $s_n\rightarrow\infty$ as $n\rightarrow\infty$. 
     By Lemma (\ref{lhf}) (ii), we have that $$\int_{\mathbb{Z}^3}(R_\alpha \ast F(s_nu_n))F(s_nu_n)\,d\mu\geq s_n^\theta\int_{\mathbb{Z}^3}\left(R_\alpha \ast F\left(u_n\right)\right)F\left(u_n\right)\,d\mu.$$
Since $\|u_n\|=1$, one gets easily that
$$\int_{\mathbb{Z}^3}\left(R_\alpha \ast F\left(u_n\right)\right)F\left(u_n\right)\,d\mu\leq C.$$
Then it follows from (i) and (iii) of Lemma \ref{lgj} and   $\theta>4$ that
\begin{eqnarray*}
 0 &<& \frac{J\left(s_nu_{n}\right)}{\left\|s_nu_{n}\right\|^{4}}\\&=&\frac{1}{2\left\|s_nu_{n}\right\|^{2}}+\frac{b\left(\int_{\mathbb{Z}^{3}}\left|\nabla (s_nu_{n})\right|^{2} d x\right)^{2}}{4\left\|s_nu_{n}\right\|^{4}}-\frac{\int_{\mathbb{Z}^3}(R_\alpha \ast F(s_nu_n))F(s_nu_n)\,d\mu}{\|s_nu_{n}\|^4}\\&\leq & \frac{1}{2s_n^{2}}+\frac{b}{4}-s_n^{\theta-4}\int_{\mathbb{Z}^3}\left(R_\alpha \ast F\left(u_n\right)\right)F\left(u_n\right)\,d\mu\rightarrow-\infty,  \end{eqnarray*}
 which is a contradiction. Hence  $\{s_n\}$ is bounded. By the boundedness of  $\{s_n\}$, up to a subsequence, there exists $s_0>0$ such that $s_n\rightarrow s_0$ and $\widehat{m}(u_n)\rightarrow s_0 u$. Since $\mathcal{N}$ is closed, $s_0u\in\mathcal{N}$. This implies that $s_0=s_u$. Hence we have
 $$s_{u_n}\rightarrow s_u$$
 and
 $$\widehat{m}(u_n)\rightarrow s_0u=s_uu=\widehat{m}(u).$$
Therefore $s$ and $\widehat{m}$ are continuous.
\
\

 (ii) Clearly, $m$ is continuous. For any $u\in\mathcal{N}$, let $\bar{u}=\frac{u}{\|u\|}$, then $\bar{u}\in S.$ Since $u=\|u\|\bar{u}$ and $s_{\bar{u}}$ is unique, we obtain $s_{\bar{u}}=\|u\|.$ Hence $m(\bar{u})=s_{\bar{u}}\bar{u}=u\in\mathcal{N}$. This implies that $m$ is surjective. Now we prove $m$ is injective. Let $u_1,u_2\in S$ and $m(u_1)=m(u_2)$. Then $s_1u_1=s_2u_2,$ and hence $s_1=s_2$. This means $u_1=u_2$. Hence $m$ is injective and has an inverse mapping $m^{-1}:\mathcal{N}\rightarrow S$ with $m^{-1}(u)=\frac{u}{\|u\|}.$ Then $m^{-1}(m(u))=u=id(u)$ for any $u\in S$.
    
\end{proof}



 Now we set $$c:=\inf\limits_{\mathcal{N}}J>0,$$
$$c_1:=\inf\limits_{u\in H \backslash\{0\}} \max _{s>0} J(su),$$
and
$$c_2:=\inf\limits_{\gamma\in\Gamma_2}\max\limits_{s\in [0,1]} J(\gamma(s)),$$ 
where 
$$\Gamma_2=\{\gamma\in C([0,1],H):\gamma(0)=0, J(\gamma(1))<0\}.$$ 

\begin{lm}\label{jb}
Let $\left(h_{1}\right)$ and $\left(f_{1}\right)$-$\left(f_{4}\right)$ hold. Then $c_1=c_2=c>0.$   
\end{lm}
\begin{proof}
    We first prove $c_1=c.$ By Lemma \ref{lgj} (i), there exists a unique $s_u>0$ such that
$J(s_{u}u)=\max\limits_{s>0} J(s u)$. Then 
$$c_1=\inf\limits_{u\in H \backslash\{0\}} \max _{s>0} J(su)=\inf\limits_{u\in H \backslash\{0\}}J(s_{u}u)= \inf\limits_{u\in\mathcal{N}} J(u)=c.$$

Now we prove $c_1\geq c_2.$ By (\ref{ud}), for any $u\in H\backslash\{0\},$ there exists a large $s_0>0$ such that $J(s_0 u)<0.$ Define
$$
\begin{aligned}
\gamma_0: [0,1] & \rightarrow H, \\
 s & \mapsto ss_0u.
\end{aligned}
$$ 
Since $\gamma_0(0)=0$ and $J(\gamma_0(1))<0$, we have $\gamma_0\in\Gamma_2$. Then for any $u\in H\backslash\{0\},$ we get that
$$\max _{s>0} J(su)\geq \max _{s\in[0,1]} J(ss_0u)=\max _{s\in[0,1]} J(\gamma_0(s))\geq \inf\limits_{\gamma\in\Gamma_2}\max\limits_{s\in [0,1]} J(\gamma(s)),$$
which implies that
$c_1\geq c_2.$

In the following, we prove $c_2\geq c.$  By Lemma \ref{lgj} (i),  for any $u\in H\backslash\{0\},$ there exists a unique $s_u>0$ such that $s_uu\in\mathcal{N}.$ Then we can separate $H$ into two components, namely $H=H_1\cup H_2$, where $H_1=\{u\in H: s_u\geq 1\}$ and $H_2=\{u\in H: s_u<1\}$.
We claim that each $\gamma\in\Gamma_2$ has to cross $\mathcal{N}$. In fact, one gets easily that $\gamma(t)$ and $0$ belong to $H_1$ for $s$ small enough. We only need to prove $\gamma(1)\in H_2$. 
Let $$G(s)=J(s\gamma(1)),\quad  s\geq 0.$$ Clearly $G(0)=0$ and $G(1)<0$. Similar arguments to (\ref{jv}), we get that
$G(t)>0$ for $s>0$ small enough. Hence there exists $s_{\gamma(1)}\in (0,1)$ such that $\max\limits_{s\geq 0}G(s)=J(s_{\gamma(1)}\gamma(1))$. Hence $\gamma(1)\in H_2.$ By the continuity of $s$ in Lemma \ref{lg}, we get that each $\gamma\in\Gamma_2$ has to cross $\mathcal{N}$. Then for any $\gamma\in \Gamma_2$, there exists $t_0\in (0,1)$ such that $\gamma(t_0)\in\mathcal{N}.$
As a consequence,
$$\inf\limits_{u\in \mathcal{N}}J(u)\leq J(\gamma(t_0))\leq \max\limits_{s\in [0,1]} J(\gamma(s)),$$
which implies that $c\leq c_2.$ Therefore, we have $c_1=c_2=c.$

\end{proof}

\
\

{\bf Proof of Theorem \ref{t1} } By Lemma \ref{lm} and Lemma \ref{ln}, one sees that $J$ satisfies the geometric structure and $(PS)_{c_2}$ condition. Then by the mountain pass theorem, there exists $u\in H$ such that $J(u)=c_2$ and $J'(u)=0$. Then it follows from Lemma \ref{jb} that $c_2=c>0$. Hence $u\neq 0$ and $u\in \mathcal{N}.$ The proof is completed. \qed

\
\

\section{ Proof of Theorem \ref{t2}}
In this section, we prove the existence of ground state solutions to the equation (\ref{aa}) under the conditions $(h_1)$ and $(h_3)$ on $V$. 

As we see, the condition $(h_2)$ ensures a compact embedding, see Lemma \ref{lgg}, while the condition $(h_3)$ leads to the lack of compactness. Moreover, since we  only assume that $f$ is continuous, $\mathcal{N}$ is not a $C^1$-manifold. This makes that we cannot use the Ekeland variational principle on $\mathcal{N}$ directly. Note that Lemma \ref{lgj} and Lemma \ref{lg} still hold. Hence we shall follow the lines of Hua and Xu \cite{HX} to prove this theorem. 

We show that $\Psi$ (see below) is of  class $C^{1}$ and there is a one-to-one correspondence between critical points of $\Psi$ and nontrivial critical points of $J$. The proof of the following lemma is similar to \cite{HX,SW}. For
completeness, we present the proof in the context.

\begin{lm}\label{lgk}
  Let $\left(h_{1}\right)$ and $\left(f_{1}\right)$-$\left(f_{4}\right)$ hold. Define the functional
  $$
  \begin{aligned}
\Psi: S &\rightarrow \mathbb{R},\\
w & \mapsto \Psi(w)=J(m(w)).
\end{aligned}
$$ 
Then
\begin{itemize}
    \item [(i)] $\Psi(w) \in C^{1}(S, \mathbb{R})$ and $$
\langle \Psi^{\prime}(w),z\rangle=\|m(w)\|\left\langle J^{\prime}(m(w)), z\right\rangle, \quad z \in T_{w}(S)=\{v\in H:(w, v)=0\} .
$$
\item [(ii)] $\left\{w_{n}\right\}$ is a Palais-Smale sequence for $\Psi$ if and only if $\left\{m\left(w_{n}\right)\right\}$ is a Palais-Smale sequence for $J$.

\item[(iii)] 
$w\in S$ is a critical point of $\Psi$ if and only if $m(w) \in \mathcal{N}$ is a nontrivial critical point of $J$. Moreover, the corresponding critical values of $\Psi$ and $J$ coincide and $
\inf\limits_{S} \Psi=\inf\limits_{\mathcal{N}}J.
$
\end{itemize}
 
\end{lm} 
\begin{proof}
  (i) Define the functional 
  $$
  \begin{aligned}
\widehat{\Psi}: H\backslash\{0\}&\rightarrow \mathbb{R},\\
w & \mapsto \widehat{\Psi}(w)=J(\hat{m}(w)).
\end{aligned}
$$Since $J\in C^1(H,\mathbb{R})$ and $\widehat{m}(w)=s_ww$ is a continuous map, we have
  \begin{eqnarray*}
      \langle \widehat{\Psi}'(w),z\rangle&=&\frac{d}{dt}\mid_{t=0}\widehat{\Psi}(w+tz)\\&=&\frac{d}{dt}\mid_{t=0}J(\widehat{m}(w+tz))\\&=&J'(\widehat{m}(w+tz))\mid_{t=0}\cdot \frac{d}{dt}\mid_{t=0}\widehat{m}(w+tz)\\&=&J'(\widehat{m}(w))s_wz\\&=&s_w\langle J'(\widehat{m}(w)),z\rangle\\&=&\frac{\|\widehat{m}(w)\|}{\|w\|}\langle J'(\widehat{m}(w)),z\rangle.
  \end{eqnarray*}
Note that $\Psi=\widehat{\Psi}\mid_{S}$ and $m=\widehat{m}\mid_{S}$. Hence the result follows from the above equality.
  \
  \
  
 (ii) Denote $$\psi(u)=\frac{1}{2}\|u\|^2,\quad u\in H.$$ Clearly, $\psi\in C^1(H,\mathbb{R})$, and for any $v\in H$, $$\langle \psi'(u),v\rangle=(u,v).$$
 Hence $\psi'$ is bounded on finite sets and $\langle \psi'(w),w\rangle=1$ for all $w\in S.$
Then for any $w\in S$, we have $H=T_w(S)\oplus \mathbb{R}w,$ and the projection $$H\rightarrow T_w(S): z+tw\mapsto z$$ has uniformly bounded norm with respect to $w\in S$. In fact, note that $\psi'$ is bounded on finite sets and $\langle \psi'(w),(z+tw)\rangle=t$. If $\|z+tw\|=1,$ then $|t|\leq C$. Hence \begin{equation}\label{ue}
 \|z\|\leq |t|+\|z+tw\|\leq (1+C)\|z+tw\|,\quad w\in S, z\in T_w(S)~\text{and}~t\in \mathbb{R}.   
\end{equation}

Let $u:=m(w)$. On the one hand, by (i), we have
\begin{equation}\label{uf}
\|\Psi'(w)\|=\sup_{\substack{z\in T_w(S),\\ \|z\|=1}}\langle\Psi'(w),z\rangle=\|u\|\sup_{\substack{z\in T_w(S),\\ \|z\|=1}}\langle J'(u),z\rangle.   
\end{equation}

On the other hand, since $u\in\mathcal{N}$, we have $\langle J'(u),w\rangle=\frac{1}{\|u\|}\langle J'(u),u\rangle=0$.  Then it follows from (\ref{ue}), (\ref{uf}) and (i) that
\begin{eqnarray*}
\|\Psi'(w)\|&\leq & \|u\|\|J'(u)\| \\&=& \|u\|\sup_{\substack{z\in T_w(S),t\in\mathbb{R},\\ z+tw\neq 0}}\frac{\langle J'(u),(z+tw)\rangle}{\|z+tw\|}\\&\leq&(1+C)\sup_{z\in T_w(S)\backslash\{0\}}\frac{\|u\|\langle J'(u),z\rangle}{\|z\|}\\&=&(1+C)\sup_{z\in T_w(S)\backslash\{0\}}\frac{\langle\Psi'(w),z\rangle}{\|z\|}\\&=&(1+C)\|\Psi'(w)\|.
\end{eqnarray*}
By Lemma \ref{lgj} (ii), we have $\|u\|\geq\eta>0$ for all $u\in\mathcal{N}$. Then the result follows from the previous estimate and the fact $J(u)=\Psi(w)$.

\
\

(iii) By (\ref{uf}), $\Psi'(w)=0$ if and only if $J'(u)=0$. The rest is clear.
\end{proof}

\
\

{\bf Proof of Theorem \ref{t2} }  Note that $c=\inf\limits_{S}\Psi$. Let $\left\{w_{n}\right\} \subset S$ be a minimizing sequence such that $\Psi\left(w_{n}\right) \rightarrow c$. By Ekeland's variational principle, we may assume that $\Psi^{\prime}\left(w_{n}\right) \rightarrow 0$ as $n \rightarrow \infty$. Hence $\left\{w_{n}\right\}$ is a $(PS)_c$ sequence for $\Psi.$ 

Let $u_{n}=m\left(w_{n}\right) \in \mathcal{N}$. Then it follows from Lemma \ref{lgk} that $$J\left(u_{n}\right) \rightarrow c,\qquad\text{and}\qquad J^{\prime}\left(u_{n}\right) \rightarrow 0, \quad n \rightarrow \infty.$$  By (\ref{ba}), one gets that $\left\{u_{n}\right\}$ is bounded in $H$. Then there exists $u \in H$ such that $$u_{n} \rightharpoonup u,\quad\text{in~}H, \qquad\text{and}\qquad u_n\rightarrow u,\quad\text{pointwise~in}~\mathbb{Z}^3.$$

If \begin{equation}\label{hq}
 \left\|u_{n}\right\|_{\infty} \rightarrow 0,\quad n\rightarrow\infty,   
\end{equation} 
then by Lemma \ref{lgh}, we have that $u_{n} \rightarrow 0$ in $\ell^{t}\left(\mathbb{Z}^{3}\right)$ with $t>2$. Hence  \begin{eqnarray*}
    \int_{\mathbb{Z}^3}(R_\alpha \ast F(u_n))f(u_n)u_n\,d\mu&\leq& C\left( \|u_n\|^4_{\frac{12}{3+\alpha}}+\|u_n\|^{2p}_{\frac{6p}{3+\alpha}}\right)\\&\rightarrow& 0,\quad n\rightarrow\infty.
\end{eqnarray*}
Namely
$$\int_{\mathbb{Z}^3}(R_\alpha \ast F(u_n))f(u_n)u_n\,d\mu=o_{n}(1).$$  Then
$$
\begin{aligned}
0=\left\langle J^{\prime}\left(u_{n}\right), u_{n}\right\rangle & =\left\|u_{n}\right\|^{2}+b\left(\int_{\mathbb{Z}^{3}}\left|\nabla u_{n}\right|^{2} d \mu\right)^{2}-\int_{\mathbb{Z}^3}(R_\alpha \ast F(u_n))f(u_n)u_n\,d\mu\\
& \geq\left\|u_{n}\right\|^{2}+o_{n}(1),
\end{aligned}
$$
which implies that $\left\|u_{n}\right\| \rightarrow 0$ as $n \rightarrow \infty$. This contradicts $\left\|u_{n}\right\| \geq \eta>0$ in Lemma \ref{lgj} (ii). Hence (\ref{hq}) does not hold, and hence there exists $\delta>0$ such that
\begin{equation}\label{hw}
\liminf _{n \rightarrow \infty}\left\|u_{n}\right\|_{\infty} \geq \delta>0,
\end{equation}
which implies that $u\neq 0$.
Therefore, there exists a sequence $\left\{y_{n}\right\} \subset \mathbb{Z}^{3}$ such that
\begin{equation*}
\left|u_{n}(y_{n})\right| \geq \frac{\delta}{2}.
\end{equation*}
Let $k_{n}\in \mathbb{Z}^{3}$ satisfy $\left\{y_{n}-k_{n} \tau\right\} \subset \Omega$, where $\Omega=[0, \tau)^{3}$. By translations, let $v_{n}(y):=u_{n}\left(y+k_{n} \tau\right)$. Then for any $v_{n}$,
\begin{equation*}
\left\|v_{n}\right\|_{l^{\infty}(\Omega)} \geq\left|v_{n}\left(y_{n}-k_{n} \tau\right)\right|=\left|u_{n}(y_{n})\right| \geq \frac{\delta}{2}>0.
\end{equation*}
Since $V(x)$ is $\tau$-periodic, $J$ and $\mathcal{N}$ are invariant under the translation, we obtain that $\left\{v_{n}\right\}$ is also a $(PS)_c$ sequence for $J$ and bounded in $H$. Then there exists $v\in H$ with $v\neq 0$ such that $$v_{n} \rightharpoonup v,\quad\text{in~}H, \qquad\text{and}\qquad v_n\rightarrow v,\quad \text{pointwise~in~}\mathbb{Z}^3.$$
Now, we prove that $v$ is a critical point of $J$. We assume that there exists a nonnegative constant $A$ such that $\int_{\mathbb{Z}^{3}}|\nabla v_{n}|^{2} d \mu \rightarrow A$ as $n \rightarrow \infty$. Note that
$$
\int_{\mathbb{Z}^{3}}|\nabla v|^{2} d \mu \leq \liminf _{n \rightarrow \infty} \int_{\mathbb{Z}^{3}}\left|\nabla v_{n}\right|^{2} d \mu=A.
$$
We claim that
$$
\int_{\mathbb{Z}^{3}}|\nabla v|^{2} d \mu=A.
$$
Arguing by contradiction, we assume that $\int_{\mathbb{Z}^{3}}|\nabla v|^{2} d \mu<A$. For any $\phi\in C_c(\mathbb{Z}^3),$ we have $\langle J^{\prime}\left(v_{n}\right),\phi\rangle=o_n(1)$, namely
 \begin{equation}\label{hr}
\int_{\mathbb{Z}^{3}}\left(a \nabla v_{n} \nabla \varphi+V(x) v_{n} \varphi\right)\,d \mu+b \int_{\mathbb{Z}^{3}}\left|\nabla v_{n}\right|^{2} \,d\mu \int_{\mathbb{Z}^{3}} \nabla v_{n} \nabla \varphi\,d \mu-\int_{\mathbb{Z}^3}(R_\alpha \ast F(v_n))f(v_n)\phi\,d\mu=o_{n}(1).
\end{equation}
Let $n\rightarrow\infty$ in (\ref{hr}) and by Lemma \ref{lhj}, we get that
\begin{equation}\label{ht}
\int_{\mathbb{Z}^{3}}(a \nabla v \nabla \varphi+V(x) v \varphi)\, d \mu+b A \int_{\mathbb{Z}^{3}} \nabla v \nabla \varphi\,d \mu-\int_{\mathbb{Z}^3}(R_\alpha \ast F(v))f(v)\phi\,d\mu=0.
\end{equation}
Since $C_c(\mathbb{Z}^3)$ is dense in $H$,  (\ref{ht}) holds for any $\phi\in H$. Let $\phi=v$ in (\ref{ht}), then we have 
$$
\begin{aligned}
\langle J'(v),v\rangle & =\int_{\mathbb{Z}^{3}}\left(a|\nabla v|^{2}+V(x) v^{2}\right) d \mu+b\left(\int_{\mathbb{Z}^{3}}|\nabla v|^{2} d \mu\right)^{2}-\int_{\mathbb{Z}^3}(R_\alpha \ast F(v))f(v)v\,d\mu\\&<\int_{\mathbb{Z}^{3}}\left(a|\nabla v|^{2}+V(x) v^{2}\right) d \mu+b A\int_{\mathbb{Z}^{3}}|\nabla v|^{2} d \mu-\int_{\mathbb{Z}^3}(R_\alpha \ast F(v))f(v)v\,d\mu\\&=0.
\end{aligned}
$$
Let $$h(s)=\left\langle J^{\prime}(sv), s v\right\rangle,\quad s>0.$$ 
Then $h(1)=\left\langle J^{\prime}(v), v\right\rangle<0$.

By (\ref{aj}), we get that
$$\int_{\mathbb{Z}^3}(R_\alpha \ast F(sv))f(sv)sv\nonumber\,d\mu
\leq\varepsilon s^4\|v\|^4+C_\varepsilon s^{2p}\|v\|^{2p}.$$
Then for $s>0$ small enough, 
\begin{eqnarray}\label{he}
h(s)\nonumber&=&\left\langle J^{\prime}(sv), s v\right\rangle\nonumber\\&=& s^2 \|v\|^2+s^4 b\left(\int_{\mathbb{Z}^{3}}|\nabla v|^{2} d \mu\right)^{2}-\int_{\mathbb{Z}^3}(R_\alpha \ast F(sv))f(sv)sv\,d\mu\nonumber\\&\geq& s^2 \|v\|^2-\varepsilon s^4\|v\|^4-C_\varepsilon s^{2p}\|v\|^{2p}\nonumber\\&>&0.
\end{eqnarray}
Hence, there exists $s_{0} \in(0,1)$ such that $h\left(s_{0}\right)=0$, namely $ \left\langle J^{\prime}\left(s_{0} v\right), s_{0} v\right\rangle=0$. This means that $s_0v\in\mathcal{N}$, and hence $J\left(s_0 v\right)\geq c$.  By Lemma \ref{lhf}, we get that
$$\frac{1}{4}\int_{\mathbb{Z}^3}(R_\alpha \ast F(sv))f(sv)sv\,d\mu-\frac{1}{2}\int_{\mathbb{Z}^3}(R_\alpha \ast F(sv))F(sv)\,d\mu=\frac{1}{4}sg'(s)-g(s)>0,$$ and is strictly increasing with respect to $s>0$. By $(f_3)$, one has that
$$\left(\frac{1}{4}f(v)v-\frac{1}{2}F(v)\right)>\frac{1}{2}\left(\frac{2}{\theta}f(v)v-F(v)\right)\geq 0.$$
Then by Fatou's lemma, we obtain that
$$
\begin{aligned}
c & \leq J\left(s_0 v\right)=J\left(s_0 v\right)-\frac{1}{4}\left\langle J^{\prime}\left(s_0 v\right), s_0 v\right\rangle \\
& =\frac{s_0^{2}}{4} \|v\|^2+\frac{1}{4}\int_{\mathbb{Z}^3}(R_\alpha \ast F(s_0v))f(s_0v)s_0v\,d\mu-\frac{1}{2}\int_{\mathbb{Z}^3}(R_\alpha \ast F(s_0v))F(s_0v)\,d\mu\\
& <\frac{1}{4}\|v\|^2+\frac{1}{4}\int_{\mathbb{Z}^3}(R_\alpha \ast F(v))f(v)v\,d\mu-\frac{1}{2}\int_{\mathbb{Z}^3}(R_\alpha \ast F(v))F(v)\,d\mu\\
& \leq \liminf _{n \rightarrow \infty}\left[\frac{1}{4}\|v_n\|^2+\frac{1}{4}\int_{\mathbb{Z}^3}(R_\alpha \ast F(v_n))f(v_n)v_n\,d\mu-\frac{1}{2}\int_{\mathbb{Z}^3}(R_\alpha \ast F(v_n))F(v_n )\,d\mu\right] \\
& =\liminf _{n \rightarrow \infty}\left[J\left(v_{n}\right)-\frac{1}{4}\left\langle J^{\prime}\left(v_{n}\right), v_{n}\right\rangle\right] \\
& =c.
\end{aligned}
$$
This is a contradiction. Hence, \begin{equation*}
\int_{\mathbb{Z}^{3}}\left|\nabla v_{n}\right|^{2} d \mu \rightarrow \int_{\mathbb{Z}^{3}}|\nabla v|^{2} d \mu=A.
\end{equation*}
Then it follows from (\ref{hr}) and (\ref{ht}) that $J^{\prime}(v)=0$, and hence $v \in \mathcal{N}$. 
It remains to prove that $J(v)=c$. In fact, by Fatou's lemma, we obtain that
$$
\begin{aligned}
c & \leq J(v)-\frac{1}{4}\left\langle J^{\prime}(v), v\right\rangle \\
& =\frac{1}{4}\|v\|^2+\frac{1}{4}\int_{\mathbb{Z}^3}(R_\alpha \ast F(v))f(v)v\,d\mu-\frac{1}{2}\int_{\mathbb{Z}^3}(R_\alpha \ast F(v))F(v)\,d\mu\\
& \leq \liminf _{n \rightarrow \infty}\left[\frac{1}{4}\|v_n\|^2+\frac{1}{4}\int_{\mathbb{Z}^3}(R_\alpha \ast F(v_n))f(v_n)v_n\,d\mu-\frac{1}{2}\int_{\mathbb{Z}^3}(R_\alpha \ast F(v_n))F(v_n)\,d\mu\right] \\
& =\liminf _{n \rightarrow \infty}\left[J\left(v_{n}\right)-\frac{1}{4}\left\langle J^{\prime}\left(v_{n}\right), v_{n}\right\rangle\right] \\
& =c .
\end{aligned}
$$
Thus $J(v)=c$. The proof is completed.\qed

\
\

{\bf Declarations}

\
\

{\bf Conflict of interest:} The author declares that there are no conflicts of interests regarding the publication of
this paper.

\end{document}